\RequirePackage{filecontents}        % loading package filecontents
\documentclass[11pt]{amsart}
\usepackage{amsfonts,amssymb,amscd,amsmath,enumerate,verbatim,calc,graphicx}

\usepackage[numbers]{natbib}         % bibliography style
\def\NZQ{\mathbb}               % the font for N,Z,Q,R,C

\def\ZZ{{\NZQ Z}}
\def\RR{{\NZQ R}}
\def\CC{{\NZQ C}}

\def\FF{{\NZQ F}}

\def\frk{\mathfrak}               % font for "Fraktur"

\def\Phi{{\frk N}}
\def\opn#1#2{\def#1{\operatorname{#2}}} % to make operators
\opn\con{conv} \opn\relint{relint} \opn\vol{vol} \opn\heit{ht} \opn\supp{supp}
\opn\gr{gr}

\def\lfr{\left\lfloor}
\def\rfr{\right\rfloor}

\newtheorem{Theorem}{Theorem}[section]
\newtheorem{Lemma}[Theorem]{Lemma}
\newtheorem{Corollary}[Theorem]{Corollary}
\newtheorem{Proposition}[Theorem]{Proposition}

\newtheorem{Example}[Theorem]{Example}

\newtheorem{Definition}[Theorem]{Definition}

\newtheorem{Conjecture}[Theorem]{Conjecture}

\newtheorem*{Ack}{Acknowledgment}

\begin{document}

\title{Gorenstein polytopes with trinomial $h^*$-polynomials}

\author{Akihiro Higashitani, Benjamin Nill and Akiyoshi Tsuchiya}
\address{Akihiro Higashitani,
Department of Mathematics, Graduate School of Science, 
Kyoto University, Kitashirakawa-Oiwake cho, Sakyo-ku, Kyoto 606-8502, Japan}
\email{ahigashi@math.kyoto-u.ac.jp}
\address{Benjamin Nill, 
Department of Mathematics, 
Stockholm University, Kr\"aftriket, SE-10691, Stockholm, Sweden}
\email{nill@math.su.se}
\address{Akiyoshi Tsuchiya,
Department of Pure and Applied Mathematics,
Graduate School of Information Science and Technology,
Osaka University,
Toyonaka, Osaka 560-0043, Japan}
\email{a-tsuchiya@cr.math.sci.osaka-u.ac.jp}

\begin{abstract}
The characterization of lattice polytopes based upon information about their Ehrhart $h^*$-polynomials is a difficult open problem. In this paper, we finish the classification of lattice polytopes whose $h^*$-polynomials satisfy two properties: they are palindromic (so the polytope is Gorenstein) and they consist of precisely three terms. This extends the classification of Gorenstein polytopes of degree two due to Batyrev and Juny. The proof relies on the recent characterization of Batyrev and Hofscheier of empty lattice simplices whose $h^*$-polynomials have precisely two terms. Putting our theorem in perspective, we give a summary of these and other existing results in this area.
\end{abstract}

\maketitle

\section{Introduction}

\subsection{Basic notions and terminology}

Let us start by setting up notation and recalling the main objects. For an introduction to Ehrhart theory, we refer to \cite{BeckRobins}.% or \cite{HibiRedBook}. 

Let $\Delta \subset \RR^d$ be a {\em lattice polytope} of dimension $d$ (i.e., $\Delta$ is a full-dimensional convex polytope in $\RR^d$ with vertices in $\ZZ^d$). Throughout the paper, lattice polytopes are identified if they are {\em isomorphic} via an affine lattice-preserving transformation. The {\em Ehrhart series} of $\Delta$ 
$$\text{Ehr}_\Delta(t)=1+\sum_{k = 1}^\infty \sharp (k \Delta \cap \ZZ^d) t^k$$ 
is a rational function of the form 
$$\text{Ehr}_\Delta(t)=\frac{h_0^*+h_1^*t+\cdots+h_d^*t^d}{(1-t)^{d+1}},$$ 
where the polynomial $h_0^*+h_1^*t+\cdots+h_d^*t^d$ appearing in the numerator has nonnegative integer coefficients. 
We call the polynomial $h^*_\Delta(t)=h_0^*+h_1^*t+\cdots+h_d^*t^d$ the {\em $h^*$-polynomial} of $\Delta$. 
The positive integer $\sum_{i=0}^d h_i^*$ is the {\em normalized volume} of $\Delta$, denoted by $\vol(\Delta)$; it is equal to $d!$ times 
the usual Euclidean volume of $\Delta$. Note that there are only finitely many lattice polytopes of fixed normalized volume \cite{LZ91}. The highest possible coefficient $h^*_d$ equals the number of interior lattice points of $\Delta$. The maximal integer $s$ such that $h^*_s \not=0$ is called the {\em degree} of $\Delta$. We also have the equality $h^*_1 =  \sharp (\Delta \cap \ZZ) - d - 1$. A lattice simplex $\Delta$ is an {\em empty simplex} if 
$\Delta$ contains no lattice points except for its vertices, equivalently, $h^*_1=0$. Empty simplices appear naturally in singularity theory \cite{BBBK11} and optimization \cite{Seb99}.

There are two well-known higher-dimensional constructions of lattice polytopes. Let $\con(S)$ denote the convex hull of a subset $S \subset \RR^d$. 
For a lattice polytope $\Delta \subset \RR^d$, we can construct a new lattice polytope 
$$\Delta'=\con(\Delta \times \{0\}, (0,\ldots,0,1)) \subset \RR^{d+1}$$ 
of dimension $d+1$. This polytope $\Delta'$ is called the {\em lattice pyramid} over $\Delta$. We often use lattice pyramid shortly for a lattice polytope that has been obtained by successively taking lattice pyramids. Note that the $h^*$-polynomial does not change under lattice pyramids \cite{Bat06}. 
We also define a lattice polytope $\Delta$ to be a {\em Cayley polytope} of $\Delta_1, \ldots, \Delta_n \subset \RR^t$, if $\Delta$ is isomorphic to 
\[\con(e_1 \times \Delta_1, \ldots, e_n \times \Delta_n) \subset \RR^n \times \RR^t,\]
where $e_1, \ldots, e_n$ is the standard lattice basis of $\RR^n$, see e.g. \cite{BN07}.

\subsection{Background}

One of the directions in Ehrhart theory is to characterize $h^*$-poly\-nomials that have an especially simple form and to classify all lattice polytopes with these $h^*$-polynomials. The motivation is that such results 
help to understand the restrictions that this important invariant imposes on the structure of a lattice polytope and to learn what to expect in more general situations. 
In order to put our main theorem in this paper in perspective, we will present some of the existing results in this area. \medskip

%Terminology that is only used at this point of the paper will not be defined in all the detail. 

\begin{enumerate}
\item {\em Small dimensions}: Let us describe what is known about $h^*$-polynomials of small-dimensional lattice polytopes. In dimension $d=1$, for a given lattice interval of length $a+1$, we have $h^*(t)=1 + a t$. In dimension $2$, the $h^*$-polynomials of lattice polygons have been classified by Scott \cite{Scott}. It holds that $1+at +bt^2$ with $a,b \in \ZZ_{\ge0}$ is the $h^*$-polynomial of a lattice polygon if and only if 
\begin{itemize}
\item $b=0$ (i.e., $\Delta$ has no interior lattice points), or
\item $b=1$ and $a=7$ (here, $\Delta$ is isomorphic to $\con((0,0),(3,0),(0,3))$), or
\item $b \ge 1$ and $b \le a \le 3b+3$.
\end{itemize}
The upper bound in the last point is often refered to as Scott's theorem. We refer to \cite{Haase} for a thorough discussion. 

\smallskip

In dimension $3$ there are currently only partial results known. The arguably most significant one is White's theorem \cite{White}: a three-dimensional lattice simplex is empty (i.e., $h^*_1=0$) 
if and only if it is the Cayley polytope of two empty line segments in $\RR^2$. Recently, all three-dimensional lattice polytopes with at most $6$ lattice points (i.e., $h^*_1\le 2$) have been classified \cite{Santos1,Santos2}.\medskip

\item {\em Small degree}: It is natural to take the degree of the $h^*$-polynomial as a measure of complexity. Any degree zero lattice polytope is a unimodular simplex (i.e., the convex 
hull of affine lattice basis). For degree one, taking lattice pyramids over lattice intervals yields that 
any $h^*$-polynomial $1+at$ with arbitrary $a \in \ZZ_{\ge 0}$ is possible. Lattice polytopes of degree one are completely classified \cite{BN07}:
\begin{itemize}
\item Lattice pyramids over $\con((0,0),(2,0),(0,2))$, or 
\item Cayley polytopes of line intervals in $\RR^1$.
\end{itemize}

\smallskip

Lattice polytopes of degree at most two are not yet completely classified. However, their $h^*$-polynomials are known \cite{Treutlein,Henk}. A polynomial $1+at+bt^2$ with $a,b \in \ZZ_{\ge0}$ is the $h^*$-polynomial of a lattice polytope (in some dimension) if and only if 
\begin{itemize}
\item $b=0$, or 
\item $b=1$ and $a=7$\\
(here, $\Delta$ is isomorphic to a lattice pyramid over $\con((0,0),(3,0),(0,3))$), or
\item $b \ge 1$ and $a \le 3b+3$.
\end{itemize}
Note how close this is to the characterization in dimension two above. It follows from the proof in \cite{Henk} that any such polynomial can be given by the $h^*$-polynomial of a lattice polytope in dimension three.\medskip

\item {\em Small number of monomials}: An even more general problem is to consider the number of terms in the $h^*$-polynomial. Batyrev and Hofscheier \cite{BatyrevHof1,BatyrevHof2} have recently classified 
all lattice polytopes whose $h^*$-polynomials are binomials, i.e., of the form $h^*(t)=1+a t^k$. Let $\Delta$ be a $d$-dimensional lattice polytopes with such a binomial $h^*$-polynomial. Since the degree one case $k=1$ is known, let $k \ge 2$. Hence, $h^*_1=0$ implies that $\Delta$ is an empty simplex. It can be observed \cite[Prop.1.5]{BatyrevHof2} that $d \ge 2k-1$. Let $d=2k-1$. In this case, it is proven in \cite{BatyrevHof1} that $\Delta$ has $h^*$-polynomial $1 + a t^k$ (with $a \ge 1$) if and only if $\Delta$ is a Cayley polytope of $k$ empty line segments in $\RR^k$. Note that for $d=3$ and $k=2$ this recovers White's theorem. In particular, one sees from \cite[Example 2.2]{BatyrevHof2} that any $a \in \ZZ_{\ge 1}$ and $k \in \ZZ_{\ge 1}$ is possible for an $h^*$-polynomial of the form $1+at^k$. The reader might notice the analogy with the degree one case above.

\smallskip

For $d \ge 2k$, we are in an exceptional situation. Let us consider only $h^*$-polynomials of lattice polytopes that are not lattice pyramids (otherwise, by what we've just seen, any $1+at^k$ can appear). Note that since $\Delta$ is a simplex, it follows from \cite{Nil08} that $d \le 4k-2$. Now, the following characterization can be deduced from the results in \cite{BatyrevHof2}: $1 + a t^k$ (with $a \in \ZZ_{\ge 1}$) is the $h^*$-polynomial of a $d$-dimensional lattice polytope $\Delta$ 
with $d \ge 2k$ where $\Delta$ is not a lattice pyramid if and only if 
%\item $a = \frac{4k}{4k-d-1} - 1$ and $\frac{k}{4k-d-1}$ is a power of two, or 

\centerline{$a = \frac{2kp}{d+1-p(d+1-2k)} - 1$ and $\frac{2k}{d+1-p(d+1-2k)}$ is a power of a prime $p$.}

It is not hard to see that this implies $p \le k$, in particular, $\vol(\Delta) = a+1 < 2k^2$. 
Hence, there are only finitely many non-lattice-pyramid lattice polytopes with binomial $h^*$-polynomials for given $k$ and arbitrary $d \ge 2k$. They are completely classified by Batyrev and Hofscheier \cite{BatyrevHof2}. It turns out that they are uniquely determined by their $h^*$-polynomial. As their results are the key ingredients in our proof we will describe them in more detail below (see \ref{batyrevhofsection}). \medskip

\item {\em Palindromic $h^*$-polynomials:} A polynomial $\sum_{i=0}^s a_i x^i$ (with $a_s \not=0$) is {\em palindromic}, if $a_i = a_{s-i}$ for $i=0, \ldots, s$. A lattice polytope $\Delta$ is called 
{\em Gorenstein}, if it has palindromic $h^*$-polynomial. Equivalently, the semigroup algebra associated to the cone over $\Delta$ is a Gorenstein algebra. Gorenstein polytopes are of interest in combinatorial commutative algebra, mirror symmetry, and tropical geometry (we refer to \cite{BatyrevJuny,BN08,Joswig}). 
In each dimension, there exist only finitely many Gorenstein polytopes. Any Gorenstein polytope has a dilate that is a reflexive polytope (in the sense of Batyrev \cite{Bat94}). They are known up to dimension $4$ \cite{KS00}. For fixed degree, there exist only finitely many Gorenstein polytopes 
that are not lattice pyramids \cite{HNP09}. They have been completely classified by Batyrev and Juny up to degree two \cite{BatyrevJuny}. In particular, their results imply that a polynomial $1+(m-2)t+t^2$ with $m \in \ZZ_{\ge 2}$ is the $h^*$-polynomial of a $d$-dimensional lattice polytope that is not a lattice pyramid if and only if 
\begin{itemize}
\item $d=2$ and $3 \le m \le 9$, or
\item $d=3$ and $2 \le m \le 8$, or
\item $d=4$ and $3 \le m \le 6$, or 
\item $d=5$ and $m=4$.
\end{itemize}
\end{enumerate}

Finally, let us mention that there also exist classification results of lattice polytopes of small normalized volume \cite{SmallVol}, of lattice polytopes with prime normalized volume \cite{PrimeVol}, as well as of lattice polytopes that have so-called shifted symmetric $h^*$-polynomials \cite{ShiftedSym}. 

\subsection{Classification of palindromic $h^*$-trinomials}

The main result (Theorem~\ref{main}) of this paper finishes the complete classification of all lattice polytopes that are not lattice pyramids and whose $h^*$-polynomial is palindromic and has precisely three terms. In the case of degree two, this was already done by Batyrev and Juny \cite{BatyrevJuny}. Here, we only consider the case when the degree is strictly larger than two. In this situation, the lattice polytope is necessarily an empty simplex, and we can apply methods and results of Batyrev and Hofscheier \cite{BatyrevHof1,BatyrevHof2}. Since the precise formulation of Theorem~\ref{main} needs some more notation, let us describe here only two immediate consequences. First, the complete characterization of palindromic $h^*$-trinomials:

\begin{Corollary}\label{coro}
Let $d \geq 2$, $m \geq 2$ and $k \geq 1$ be integers. 
The polynomial $1+(m-2)t^k + t^{2k}$ is the $h^*$-polynomial of 
a lattice polytope of dimension $d$ if and only if the integers $k,m,d$ satisfy one of the following conditions: 
\begin{itemize}
\item $k=1$, $3 \leq m \leq 9$ and $d = 2$;
\item $k=1$, $2 \leq m \leq 9$ and $d \geq 3$; 
\item $k \geq 2$, $m \in \{3,4,6,8\}$ and $d \geq 3k-1$; 
\item $k=2^{\ell-3}a$, $m = 2^\ell$ and $d \geq 4k-1$, where $a \geq 1$ and $\ell \geq 4$; 
\item $k=3^{\ell-2}a$, $m = 3^\ell$ and $d \geq 3k-1$, where $a \geq 1$ and $\ell \geq 3$. 
\end{itemize}
\end{Corollary}

The case $k=1$ was already known, as described in (1) and (2) of the previous section. 

Secondly, Theorem~\ref{main} implies the following uniqueness result:

\begin{Corollary}A lattice simplex $\Delta$ that is not a lattice pyramid is uniquely determined by its dimension and its $h^*$-polynomial if it is of the form $1+(m-2)t^k+t^{2k}$ with $k \ge 2$.
\end{Corollary}

Let us note that for $k \ge 2$ any of these lattice simplices that are not lattice pyramids have dimension $3k-1$ or $4k-1$, see Theorem~\ref{main}.

\subsection{Future work}

It is known \cite{HNP09} that there exists a function $f$ in terms of the degree $k$ and the leading coefficient $b$ of an $h^*$-polynomial of a lattice polytope $\Delta$ such that 
$\vol(\Delta) \le f(b,k)$. In the situation of Corollary~\ref{coro} (where $b=1$) one observes that $\Delta$ satisfies $m \le 9 k$. 
In other words, \[\vol(\Delta) \le \frac{9}{2} \deg(\Delta).\]
Moreover, equality implies $k=1$ and so as described in (2) above $\Delta$ is isomorphic to a lattice pyramid over $\con((0,0),(3,0),(0,3))$. Now, 
having seen how Scott's theorem could be generalized from dimension two to degree two \cite{Treutlein}, we make the following guess about a more general class of $h^*$-trinomials:

\begin{Conjecture} Let $\Delta$ be a lattice polytope with $h^*$-polynomial $1 + a t^k + b t^{2k}$ and $b\ge 2$. Then 
$a+b+1 \le (4b+4) k$, or equivalently, 
\[\vol(\Delta) \le \frac{4b+4}{2} \deg(\Delta).\]
\end{Conjecture}

\subsection{Organization of the paper}

In Section~2 we recall the notation and results by Batyrev and Hofscheier. In Section~3 we present and prove the main result of this paper (Theorem~\ref{main}): the classification of Gorenstein polytopes with $h^*$-trinomials of degree $\ge 3$.

\smallskip

\begin{Ack}{\rm We would like to thank Alexander Kasprzyk for discussion. 
The first author is partially supported by a JSPS Fellowship for Young Scientists and by JSPS Grant-in-Aid for Young Scientists (B) $\sharp$26800015.
The second author is partially supported by the Vetenskapsr\aa det grant NT:2014-3991.

}
\end{Ack}

\section{The approach by Batyrev and Hofscheier}

Let us describe the notions used by Batyrev and Hofscheier in \cite{BatyrevHof1, BatyrevHof2}.

\subsection{The correspondence to subgroups }

For a lattice simplex $\Delta \subset \RR^d$ of dimension $d$ with a chosen ordering of the vertices $v_0,v_1,\ldots,v_d \in \ZZ^d$, 
let $$\Lambda_\Delta = \left\{(x_0,\ldots,x_d) \in [0,1)^{d+1} : \sum_{i=0}^d x_i (v_i,1) \in \ZZ^{d+1}\right\}.$$ 
Then $\Lambda_\Delta$ is a subgroup of the additive group $(\RR/\ZZ)^{d+1}$, here identified with $[0,1)^{d+1}$: 
for $(x_0,\ldots,x_d) \in \Lambda_\Delta$ and $(y_0,\ldots,y_d) \in \Lambda_\Delta$, 
we define $(x_0,\ldots,x_d) + (y_0,\ldots,y_d) = (\{x_0+y_0\},\ldots,\{x_d + y_d\}) \in \Lambda_\Delta$, 
where for a real number $r$, $\{r\}$ denotes the fractional part of $r$, i.e., $\{r\}=r - \lfr r \rfr$. 
For a positive integer $j$ and $x \in \Lambda_\Delta$, we set $jx = \underbrace{x+ \cdots +x}_j$. 
We denote the unit of $\Lambda_\Delta$ by ${\bf 0}$ and the inverse of $x$ by $-x$. Note that e.g. $-(\frac{1}{2},\frac{1}{3})=(\frac{1}{2},\frac{2}{3})$. 
For $x = (x_0,\ldots,x_d) \in \Lambda_\Delta$, we define $\heit(x)=\sum_{i=0}^d x_i \in \ZZ.$ 

It is well known that the coefficients of the $h^*$-polynomial $\sum_{i=0}^d h_i^*t^i$ of the lattice simplex $\Delta$ can be computed as follows: 
$$h_i^*=\sharp \{ x \in \Lambda_\Delta : \heit(x) = i\} \text{ for } i=0, \ldots, d.$$
In particular, 
\[\sharp \Lambda_\Delta = \vol(\Delta).\]

\begin{Theorem}[{\cite[Theorem 2.3]{BatyrevHof2}}]
%The isomorphism class of a lattice simplex $\Delta$ is uniquely determined by its group $\Lambda_\Delta$. 
There is a bijection between isomorphism classes of $d$-dimensional lattice simplices with a chosen ordering of their vertices and finite subgroups of $(\RR/\ZZ)^{d+1}$. In particular, two lattice simplices $\Delta$, $\Delta'$ are isomorphic if and only if there 
exists an ordering of their vertices such that $\Lambda_\Delta=\Lambda_\Delta'$. 
\end{Theorem}

We recall the following statement which describes when the lattice simplex $\Delta$ 
is a lattice pyramid in terms of $\Lambda_\Delta$. 

\begin{Proposition}[{\cite[Lemma 2.3]{Nil08}}]\label{pyramid} 
Let $\Delta \subset \RR^d$ be a lattice simplex of dimension $d$. Then $\Delta$ is a lattice pyramid if and only if 
there is $i \in \{0,\ldots,d\}$ such that $x_i=0$ for all $(x_0,\ldots,x_d) \in \Lambda_\Delta$. 
\end{Proposition}

\subsection{The classification of lattice polytopes with binomial $h^*$-polynomials}

\label{batyrevhofsection}

We summarize results by Batyrev and Hofscheier from \cite{BatyrevHof1} and \cite{BatyrevHof2} 
which play a crucial role in our proof of Theorem \ref{main}. %For some terminologies, please consult \cite{BatyrevHof1} or \cite{BatyrevHof2}. \\

First, let us describe their generalization of White's theorem.

\begin{Theorem}[\cite{BatyrevHof1}]\label{key1}
Let $k \geq 2$ and let $\Delta$ be a lattice simplex of dimension $2k-1$ with $\vol(\Delta)=m$ 
which is not a lattice pyramid. Then the following statements are equivalent: 
\begin{itemize}
\item[(a)] the $h^*$-polynomial of $\Delta$ is $1+(m-1)t^k$; 
\item[(b)] $\Delta$ is isomorphic to the Cayley polytope 
$\Delta_1*\cdots*\Delta_k$ of empty simplices $\Delta_i \subset \RR^k$ of dimension 1; 
\item[(c)] $\Lambda_\Delta$ is cyclic and generated by $(a_1/m,(m-a_1)/m,\ldots,a_k/m,(m-a_k)/m) \in [0,1)^{2k}$ after reordering, 
where each $0<a_i \leq m/2$ is an integer which is coprime to $m$. 
\end{itemize}
\end{Theorem}
%\begin{proof}
%The implication (a) $\Longrightarrow$ (b) follows from \cite{BatyrevHof1}. 
%We give a proof of the implications (b) $\Longrightarrow$ (c) and (c) $\Longrightarrow$ (a). 
%
%\noindent
%(b) $\Longrightarrow$ (c): 
%
%\noindent
%(c) $\Longrightarrow$ (a): 
%Since each $a_i$ is coprime to $m$, we have $\{ja_i/m\}+\{j(m-a_i)/m\}=1$ for each $1 \leq i \leq k$. 
%This means that we have $\heit(y)=k$ for all $y \in \Lambda_\Delta \setminus \{{\bf 0}\}$, as desired. 
%\end{proof}

Batyrev and Hofscheier use the language of linear codes to consider the case $d > 2k-1$. A {\em linear code} over $\FF_p$ with {\em block length} $n$ is a subspace $L$ of the finite vector space $\FF^n_p$ (where $p$ is a prime). $A \in \FF_p^{r \times n}$ (an $r \times n$ matrix with entries in $\FF_p$) is the {\em generator matrix} of such an $r$-dimensional linear code $L$ if the rows of $A$ form a basis of $L$. 

\begin{Definition}{\em 
Fix a natural number $r$ and a prime number $p$. Let $n=(p^r-1)/(p-1)$ be the number of 
points in $(r-1)$-dimensional projective space over $\FF_p$. Consider the $r \times n$-matrix 
$A \in \FF_p^{r \times n}$ whose columns consist of nonzero vectors from each 1-dimensional subspace 
of $\FF_p^r$. Then $A$ is the generator matrix of the {\em simplex code} of dimension $r$ over $\FF_p$ with block length $n$.
}\end{Definition}

\begin{Theorem}[\cite{BatyrevHof2}]\label{key2}
Let $d \geq 3$ and let $\Delta$ be a lattice simplex of dimension $d$ which is not a lattice pyramid. 
Let the $h^*$-polynomial of $\Delta$ be $1+(m-1)t^k$ for some $m \geq 2$ and $1 < k <(d+1)/2$. Then there exists a prime number $p$ such that every non-trivial element of $\Lambda_\Delta$ has order $p$. In particular, $\Lambda_\Delta$ can be identified with 
$p \Lambda_\Delta \subseteq \{0, \ldots, p-1\}^{d+1}$, a linear code over $\FF_p$ with block length $d+1$. The order $m$ of $\Lambda_\Delta$ is equal to $p^r$, where the positive integer $r$ is the dimension of the linear code $p \Lambda_\Delta$. The numbers $p,d,k,r$ are related by the equation 
\begin{equation}(p^r-p^{r-1})(d+1)=2k(p^r-1).\label{eqq}\end{equation}
A generator matrix of the linear code $p\Lambda_\Delta$ is given (up to permutation of the columns) by 
the rows in the following $r \times (d+1)$-matrix:
%the generating matrix in $\FF_p^{r \times (d+1)}$ of the corresponding linear code can be written up to permutation of the columns in the form 
$$(A,\ldots,A) \text{ if $p=2$, or }\; (A, -A, \ldots, A, -A) \text{ if $p>2$},$$ 
where $A$ is the generator matrix of the $r$-dimensional simplex code over $\FF_p$ and $A$ (resp. the pair $(A,-A)$) is repeated 
$k/2^{r-2}$ (resp. $k/ p^{r-1}$) times if $p=2$ (resp. if $p>2$).
\end{Theorem}

Let us note that also the converse of the theorem holds, so the linear codes defined by the generator matrices given in the theorem correspond to lattice simplices with 
$h^*$-polynomial $1+(p^r-1)t^k$ if the numerical condition (\ref{eqq}) holds, see also \cite[Proposition~5.2]{BatyrevHof2}.

\section{The classification of lattice polytopes with palindromic $h^*$-trinomials}

\subsection{The main result}

If $B$ is a matrix, we denote by $(B,0)$ the matrix with one additional zero column.

%The goal of this paper is to give a proof of the following theorem. The case $k=1$ was already considered by Batyrev and Juny, see above. 

\begin{Theorem}\label{main}
Let $m \geq 3$ and $k \geq 2$ be integers and let $\Delta$ be a (necessarily empty) lattice simplex of dimension $d$ 
whose $h^*$-polynomial is $1+(m-2)t^k+t^{2k}$. Assume that $\Delta$ is not a lattice pyramid 
over any lower-dimensional simplex. Then the integers $k,m,d$ satisfy one of the following: 
\begin{itemize}
\item[(a)] $m \in \{3,4,6,8\}$ and $d = 3k-1$ or $m=4$ and $d=4k-1$; 
\item[(b)] $k=2^{\ell-3}a$, $m=2^\ell$ and $d=2^{\ell-1}a-1$, where $a \geq 1$ and $\ell \geq 3$ with $(a,\ell)\not=(1,3)$; 
\item[(c)] $k=3^{\ell-2}a$, $m=3^\ell$ and $d=3^{\ell-1}a-1$, where $a \geq 1$ and $\ell \geq 2$ with $(a,\ell)\not=(1,2)$. 
\end{itemize}
Moreover, in each case, a system of generators of the finite abelian group $\Lambda_\Delta$ 
is the set of row vectors of the matrix which can be written up to permutation of the columns as follows: 
\begin{itemize}
\item[(a)] 
\begin{align*}
&\left(1/3 \; 1/3 \; \cdots \; 1/3 \right) \in [0,1)^{1 \times 3k} \quad \text{{\em  in the case $m=3$}}; \\
&\left(\underbrace{1/4 \; \cdots \; 1/4}_{2k} \; \underbrace{1/2 \; \cdots \; 1/2}_k \right)
\in [0,1)^{1 \times 3k} \quad \text{{\em  in the case $m=4$ with $d=3k-1$}}; \\
&\left( \, 
\begin{array}{r@{}r@{}r r r}
\underbrace{
\begin{array}{rrr}
1/2 &\cdots &1/2 \\
1/2 &\cdots &1/2 
\end{array}
}_{2k} 
\; \underbrace{
\begin{array}{rrr}
0   &\cdots &0 \\
1/2 &\cdots &1/2 
\end{array}
}_{2k} 
\end{array}\, \right) \in [0,1)^{2 \times 4k} \quad \text{{\em  in the case $m=4$ with $d=4k-1$}}; \\
&\left(\underbrace{1/6 \; \cdots \; 1/6}_k \; \underbrace{1/3 \; \cdots \; 1/3}_k \; \underbrace{1/2 \; \cdots \; 1/2}_k \right)
\in [0,1)^{1 \times 3k} \quad \text{{\em  in the case $m=6$}}; \\
&\left( \, 
\begin{array}{r@{}r@{}r r r}
\underbrace{
\begin{array}{rrr}
1/2 &\cdots &1/2 \\
1/4 &\cdots &1/4 
\end{array}
}_k \; 
\underbrace{
\begin{array}{rrr}
0   &\cdots &0 \\
1/4 &\cdots &1/4
\end{array}
}_k \; 
\underbrace{
\begin{array}{rrr}
1/2 &\cdots &1/2 \\
1/2 &\cdots &1/2 
\end{array}
}_k 
\end{array}\, \right)
\in [0,1)^{2 \times 3k} \quad \text{{\em  in the case $m=8$}}. 
\end{align*}
\item[(b)] 
\begin{align*}
\begin{pmatrix}
(B_{\ell-1}^{(2)}, 0) &(B_{\ell-1}^{(2)}, 0) & \cdots &(B_{\ell-1}^{(2)},0) \\
1/2 & \cdots &\cdots &1/2
\end{pmatrix} \in [0,1)^{\ell \times 2^{\ell-1}a}, 
\end{align*}
where $A_{\ell-1}^{(2)} \in \{0,1\}^{(\ell-1) \times (2^{\ell-1}-1)}$ 
is the generator matrix of the simplex code over $\FF_2$ of dimension $(\ell-1)$ with block length $(2^{\ell-1}-1)$ 
and $B_{\ell-1}^{(2)} \in \left\{0,\frac{1}{2}\right\}^{(\ell-1) \times (2^{\ell-1}-1)}$ 
is the matrix all of whose entries are divided by $2$ from those of $A_{\ell-1}^{(2)}$, 
and where in above matrix $(B_{\ell-1}^{(2)},0) \in \left\{0,\frac{1}{2}\right\}^{(\ell-1) \times 2^{\ell-1}}$ is repeated $a$ times. 
\item[(c)] 
\begin{align*}
\begin{pmatrix}
(B_{\ell-1}^{(3)},-B_{\ell-1}^{(3)}, 0) &(B_{\ell-1}^{(3)},-B_{\ell-1}^{(3)}, 0) & \cdots &(B_{\ell-1}^{(3)},-B_{\ell-1}^{(3)}, 0) \\
1/3 & \cdots &\cdots &1/3
\end{pmatrix} \in [0,1)^{\ell \times 3^{\ell-1}a}, 
\end{align*}
where $A_{\ell-1}^{(3)} \in \{0,1,2\}^{(\ell-1) \times (3^{\ell-1}-1)/2}$ 
is the generator matrix of the simplex code over $\FF_3$ of dimension $(\ell-1)$ with block length $(3^{\ell-1}-1)/2$ and 
\linebreak $B_{\ell-1}^{(3)} \in \left\{0,\frac{1}{3},\frac{2}{3}\right\}^{(\ell-1) \times (3^{\ell-1}-1)/2}$ 
(resp. $-B_{\ell-1}^{(3)} \in \left\{0,\frac{2}{3},\frac{1}{3}\right\}^{(\ell-1) \times (3^{\ell-1}-1)/2}$) is the matrix all of whose entries are 
divided by $3$ from those of $A_{\ell-1}^{(3)}$ (resp. $-A_{\ell-1}^{(3)}$), 
and where in above matrix $(B_{\ell-1}^{(3)},-B_{\ell-1}^{(3)}, 0) \in \left\{0,\frac{1}{3},\frac{2}{3}\right\}^{(\ell-1) \times 3^{\ell-1}}$ is repeated $a$ times. 
\end{itemize}
\end{Theorem}

\begin{Example}
In case (b) for $k=2$ and $\ell=3$ the rows of the following matrix generate $\Lambda_\Delta$ of size $m=8$:

\begin{align*}
\begin{pmatrix}
\frac{1}{2} & 0 & \frac{1}{2} & 0 & \frac{1}{2} & 0 & \frac{1}{2} & 0\\
0 & \frac{1}{2} & \frac{1}{2} & 0 & 0 & \frac{1}{2} & \frac{1}{2} & 0\\
\frac{1}{2} & \frac{1}{2} & \frac{1}{2} & \frac{1}{2} & \frac{1}{2} & \frac{1}{2} & \frac{1}{2} & \frac{1}{2}
\end{pmatrix}
\end{align*}

In case (c) for $k=2$ and $\ell=2$ the rows of the following matrix generate $\Lambda_\Delta$ of size $m=9$:

\begin{align*}
\begin{pmatrix}
\frac{1}{3} & \frac{2}{3} & 0 & \frac{1}{3} & \frac{2}{3} & 0\\
\frac{1}{3} & \frac{1}{3} & \frac{1}{3} & \frac{1}{3} & \frac{1}{3} & \frac{1}{3}
\end{pmatrix}
\end{align*}
\end{Example}

\subsection{Preliminary results}

For the proof of Theorem \ref{main}, we prepare some lemmas. Throughout this section, let $\Delta$ be a lattice simplex of dimension $d$ 
whose $h^*$-polynomial equals $1+(m-2)t^k+t^{2k}$ with $k \geq 2$ and $m \geq 3$. Note that $\Delta$ is necessarily empty.

For $x=(x_0,\ldots,x_d) \in \Lambda_\Delta$, let $\supp(x)=\{i : x_i \not=0\}$. The following equality will be used throughout:
\[\sharp\supp(x) = \heit(x)+\heit(-x).\]

\begin{Lemma}\label{lem1}
Let $x \in \Lambda_\Delta$ be an element whose order is $n$ and let $1 \leq j \leq n-1$ be coprime to $n$. Then 
we have $\supp(x)=\supp(j x)$. Hence, 
\[\heit(x)+\heit((n-1)x)=\heit(jx)+\heit((n-j)x).\]
%\heit(g)+\heit((n-1)g) = \heit(jg)+\heit((n-j)g)$ if $j$ is coprime to $n$. 
%\begin{itemize}
%\item[(a)] 
%$\heit(g)+\heit((n-1)g) \geq \heit(jg)+\heit((n-j)g)$; 
%\item[(b)]
%$\heit(g)+\heit((n-1)g) = \heit(jg)+\heit((n-j)g)$ if $j$ is coprime to $n$. 
%\end{itemize}
\end{Lemma}

\begin{proof}
Let $i \in \supp(x)$, $x_i = \frac{a}{b} \not=0$ with $\gcd(a,b)=1$. By the definition of $n$, we observe that $b$ divides $a n$, so also $n$. Hence, $\gcd(b,j)=1$. Therefore, 
$b$ does not divide $ja$, so $i \in \supp(jx)$.
\end{proof}

\begin{Lemma}\label{x}
Let $x \in \Lambda_\Delta$ be the unique element with $\heit(x)=2k$. Then, 
\begin{itemize}
\item[(a)] for any $y = (y_0,\ldots,y_d) \in \Lambda_\Delta \setminus \{{\bf 0},\pm x\}$, 
we have $\sharp\supp(y)=2k$; 
\item[(b)] there is no integer $j$ and $y \in \Lambda_\Delta \setminus \{{\bf 0},\pm x\}$ 
such that $x=jy$. 
\end{itemize}
\end{Lemma}
\begin{proof}
(a) Since $\heit(y)=k$, $\heit(-y)=k$, and $y \not=x\not=-y$, we have 
$$2k=\heit(y)+\heit(-y)=\sharp\supp(y).$$
(b) For any integer $j$ and $y \in \Lambda_\Delta \setminus \{{\bf 0},\pm x\}$, 
since $\sharp \supp(y) = 2k$ by (a), we have $\sharp \supp(jy) \leq 2k$. 
However, by $\sharp \supp(x)>\heit(x) = 2k$, $x=jy$ never happens. 
\end{proof}

The following proposition is crucial for the proof of Theorem \ref{main}.

\begin{Proposition}\label{order}
Let $\Delta$ be a lattice simplex which is not a lattice pyramid 
whose $h^*$-polynomial is $1+(m-2)t^k+t^{2k}$ with $m \geq 3$ and $k \geq 2$. 
Let $x \in \Lambda_\Delta$ be the unique element with $\heit(x)=2k$. 
Then the order of $x$ must be $2$ or $3$ or $4$ or $6$, and up to permutation of coordinates 
$x$ is given as follows:
\begin{itemize}
\item $x=(1/2,\ldots,1/2) \in [0,1)^{4k}$ when its order is $2$; 
\item $x=(2/3,\ldots,2/3) \in [0,1)^{3k}$ when its order is $3$; 
\item $x=(\underbrace{3/4,\ldots,3/4}_{2k},\underbrace{1/2,\ldots,1/2}_k) \in [0,1)^{3k}$ when its order is $4$; 
\item $x=(\underbrace{5/6,\ldots,5/6}_k,\underbrace{2/3,\ldots,2/3}_k,\underbrace{1/2,\ldots,1/2}_k) \in [0,1)^{3k}$ 
when its order is $6$. 
\end{itemize}
In particular, the dimension of $\Delta$ is $4k-1$ if the order of $x$ is 2 and $3k-1$ otherwise.
\end{Proposition}
%\begin{proof}
%Since $\Lambda_\Delta$ is cyclic and the order of $\Lambda_\Delta$ is $m$, 
%$\Lambda_\Delta$ is isomorphic to $\ZZ/m \ZZ$ as a group. 
%
%
%Let $g=(g_0,\ldots,g_d) \in \Lambda_\Delta$ be a generator of $\Lambda_\Delta$ and 
%%Note that $\Lambda_\Delta$ can be generated by one element since it is cyclic. 
%%Moreover, since $\Delta$ is not a lattice pyramid, it follows from Proposition \ref{pyramid} that $g_i \not=0$ for all $0 \leq i \leq d$. 
%let $x=(x_0,\ldots,x_d) \in \Lambda_\Delta$ be the unique element with $\heit(x)=2k$. 
%Thus, for each $y \in \Lambda_\Delta \setminus \{x,{\bf 0}\}$, we have $\heit(y)=k$. 
%Moreover, there is $1 \leq j \leq m-1$ such that $x=jg=(\{jg_0\},\ldots,\{jg_d\})$. 
%
%
%Suppose that $2 \leq j \leq m-2$. By Lemma \ref{lem1} (a), we obtain 
%$$k=\heit(g) \geq \heit(jg)+\heit((m-j)g)-\heit((m-1)g) \geq 2k+k-k=2k.$$ 
%This implies $k \leq 0$, a contradiction. 
%
%Hence, $j=1$ or $j=m-1$. Suppose that $m \geq 7$ or $m=5$. 
%Then Lemma \ref{lem2} guarantees the existence of an integer $2 \leq \ell \leq m-2$ which is coprime to $m$. 
%By Lemma \ref{lem1} (b), we obtain 
%$$3k=\heit(g)+\heit((m-1)g) = \heit(\ell g)+\heit((m-\ell)g) = 2k,$$ implying that $k=0$, a contradiction. 
%
%Thus, $m \leq 6$ and $m \not=5$. Since $m \geq 3$, we conclude that $m$ must be 3 or 4 or 6. 
%\end{proof}

\begin{proof}
Let $m' \geq 2$ be the order of $x$. Suppose that $m'=5$ or $m' \geq 7$. Then $\phi(m') > 2$, where $\phi$ is the Eulerian $\phi$-function. 
In particular, there exists an integer $2 \leq j \leq m'-2$ which is coprime to $m'$. By Lemma \ref{lem1} and $-x \not= x$, we obtain
$$3k=\heit(x)+\heit((m'-1)x) = \heit(j x)+\heit((m'-j)x) = 2k,$$ implying that $k=0$, a contradiction. 
Thus, $m' \leq 6$ and $m' \not=5$. Hence, $m' \in \{2,3,4,6\}$.

\noindent
\underline{$m'=2$}: Then each $x_i$ is $1/2$ or 0. From $\heit(x)=2k$, 
we have $x=(\underbrace{1/2,\ldots,1/2}_{4k},\underbrace{0,\ldots,0}_s)$ after reordering. 
%where $s \geq 0$ is the number of 0's. 
Fix $y \in \Lambda_\Delta \setminus \{{\bf 0},x\}$ and let $q=\sharp(\supp(y) \setminus \supp(x))$. 
Since $\sharp \supp(x+y)=2k$ by Lemma \ref{x} (a), we have $2k=\sharp \supp(x+y) = 4k-k'+q$, 
where $k' = \sharp \{ i \in \supp(x) \cap \supp(y) : y_i=1/2\}$. Hence, $k'-q = 2k$. On the other hand, 
since $\sharp \supp(y)=2k$, we also have $k'+q \leq 2k$. Thus, $q \leq 0$, i.e., $q=0$. 
This means that $\supp(y) \subset \supp(x)$. 
Hence, if $s > 0$, then $\Delta$ is a lattice pyramid by Proposition \ref{pyramid}, a contradiction. 
Thus $s=0$ and we conclude that $x=(1/2,\ldots,1/2) \in [0,1)^{4k}$. 

\noindent
\underline{$m'=3$}: Then each $x_i$ is $1/3$ or $2/3$ or 0. It follows from $\heit(x)=2k$ and $\heit(-x)=k$ that 
$x=(\underbrace{2/3,\ldots,2/3}_{3k},\underbrace{0,\ldots,0}_s)$ after reordering. 
Fix $y \in \Lambda_\Delta \setminus \{{\bf 0},\pm x\}$ and let $q=\sharp(\supp(y) \setminus \supp(x))$. 
Since $\sharp \supp(x+y)=2k$, we have $2k=\sharp \supp(x+y)=3k-k_1+q$, 
where $k_1=\sharp \{ i \in \supp(x) \cap \supp(y) : y_i=1/3\}$. Hence, $k_1-q=k$. Similarly, 
since $\sharp \supp(2x+y)=2k$, we have $2k=\sharp \supp(2x+y)=3k-k_2+q$, 
where $k_2=\sharp \{ i \in \supp(x) \cap \supp(y) : y_i=2/3\}$. Hence, $k_2-q=k$. On the other hand, 
since $\sharp \supp(y)=2k$, we also have $k_1+k_2+q \leq 2k=k_1+k_2-2q$. Thus, $q \leq 0$, i.e., $q=0$, 
implying that $s=0$. Hence we conclude that $x=(2/3,\ldots,2/3) \in [0,1)^{3k}$. 

\noindent
\underline{$m'=4$}: Then each $x_i$ is $1/4$ or $1/2$ or $3/4$ or 0. For $j=1,2,3$, let $q_j=\sharp\{i : x_i=j/4\}$. 
Since $\heit(x)=(q_1+2q_2+3q_3)/4=2k$, $\heit(2x)=(q_1+q_3)/2=k$ and $\heit(3x)=(3q_1+2q_2+q_3)/4=k$, 
we obtain $q_1=0$, $q_2=k$ and $q_3=2k$, that is, 
$$x=(\underbrace{3/4,\ldots,3/4}_{2k},\underbrace{1/2,\ldots,1/2}_k, \underbrace{0,\ldots,0}_s)$$ after reordering. 
Fix $y \in \Lambda_\Delta \setminus \{jx : j=0,1,2,3\}$ and let $q=\sharp(\supp(y) \setminus \supp(x))$. 
Let $k_j=\sharp \{i \in \supp(x) \cap \supp(y) : x_i=3/4, y_i=j/4\}$ for $j=1,2,3$ and let 
$k'=\sharp \{i \in \supp(x) \cap \supp(y) : x_i=y_i=1/2\}$. 
Since $\sharp \supp(x+y)=\sharp \supp(2x+y)=\sharp \supp(3x+y)=2k$, we have the following: 
\begin{itemize}
\item $2k = \sharp \supp(x+y)=2k-k_1+k-k'+q$, i.e., $k+q=k_1+k'$; 
\item $2k = \sharp \supp(2x+y) \geq 2k-k_2+k'+q$, i.e., $k'+q \leq k_2$; 
\item $2k = \sharp \supp(3x+y)=2k-k_3+k-k'+q$, i.e., $k+q=k_3+k'$. 
\end{itemize}
In particular, we have $2k+3q \leq k_1+k_2+k_3+k'$. 
On the other hand, since $\sharp \supp(y) = 2k$, we have $k_1+k_2+k_3+k'+q \leq 2k$. Thus we obtain 
$$2k+4q \leq k_1+k_2+k_3+k'+q \leq 2k.$$
This means $q=0$, and thus, $s=0$. 
Hence we conclude that $x=(\underbrace{3/4,\ldots,3/4}_{2k},\underbrace{1/2,\ldots,1/2}_k) \in [0,1)^{3k}$ after reordering. 

\noindent
\underline{$m'=6$}: Then each $x_i$ is $1/6,1/3,1/2,2/3,5/6$ or 0. For $j=1,2,3,4,5$, let $q_j=\sharp\{i : x_i=j/6\}$. Then 
\begin{align*}
&\heit(x)=(q_1+2q_2+3q_3+4q_4+5q_5)/6=2k, \\
&\heit(2x)=(q_1+2q_2+q_4+2q_5)/3=k, \\ 
&\heit(3x)=(q_1+q_3+q_5)/2=k, \\
&\heit(4x)=(2q_1+q_2+2q_4+q_5)/3=k \text{ and } \\ 
&\heit(5x)=(5q_1+4q_2+3q_3+2q_4+q_5)/6=k. 
\end{align*}
Thus $q_1=q_2=0$ and $q_3=q_4=q_5=k$, that is, 
$$x=(\underbrace{5/6,\ldots,5/6}_k,\underbrace{2/3,\ldots,2/3}_k, \underbrace{1/2,\ldots,1/2}_k,\underbrace{0,\ldots,0}_s)$$ 
after reordering. 
Fix $y \in \Lambda_\Delta \setminus \{j x : j=0,1,2,3,4,5\}$ and let $q=\sharp(\supp(y) \setminus \supp(x))$. Let 
\begin{align*}
&k_j=\sharp \{i \in \supp(x) \cap \supp(y) : x_i=5/6, y_i=j/6\} \text{ for } j=1,2,3,4,5, \\
&k_j'=\sharp \{i \in \supp(x) \cap \supp(y) : x_i=2/3, y_i=j/3\} \text{ for } j=1,2, \\
&k''=\sharp \{i \in \supp(x) \cap \supp(y) : x_i=y_i=1/2\}. 
\end{align*}
Since $\sharp \supp(jx+y)=2k$ for $j=1,2,3,4,5$, 
%$\sharp \supp(x+y)=\sharp \supp(2x+y)=\sharp \supp(3x+y)=\sharp \supp(4x+y)=\sharp \supp(5x+y)=2k$, 
we have the following: 
\begin{itemize}
\item $2k = \sharp \supp(x+y)=k-k_1+k-k_1'+k-k''+q$, i.e., $k+q=k_1+k_1'+k''$; 
\item $2k = \sharp \supp(2x+y) \geq k-k_2+k-k_2'+k''+q$, i.e., $k''+q \leq k_2+k_2'$; 
\item $2k = \sharp \supp(3x+y) \geq k-k_3+k_1'+k_2'+k-k''+q$, i.e., $k_1'+k_2' + q \leq k_3+k''$; 
\item $2k = \sharp \supp(4x+y) \geq k-k_4+k-k_1'+k''+q$, i.e., $k''+q \leq k_4+k_1'$; 
\item $2k = \sharp \supp(5x+y)=k-k_5+k-k_2'+k-k''+q$, i.e., $k+q=k_5+k_2'+k''$. 
\end{itemize}
By summing up these five inequalities, we have $2k+5q \leq k_1+\cdots+k_5+k_1'+k_2'+k''$. 
On the other hand, since $\sharp \supp(y) = 2k$, we have $k_1+\cdots+k_5+k_1'+k_2'+k''+q \leq 2k$. Thus we obtain 
$$2k+6q \leq k_1+\cdots+k_5+k_1'+k_2'+k'' + q \leq 2k.$$
This means $q=0$, and thus, $s=0$. Hence we conclude that 
$$x=(\underbrace{5/6,\ldots,5/6}_k,\underbrace{2/3,\ldots,2/3}_k,\underbrace{1/2,\ldots,1/2}_k) \in [0,1)^{3k}$$ after reordering. 
\end{proof}

As a corollary of this proposition, we obtain the following: 
\begin{Corollary}\label{ccc}
Let $m \geq 3$ and $k \geq 2$ be integers. 
Let $\Delta$ be a lattice polytope with $h_\Delta^*(t)=1+(m-2)t^k+t^{2k}$. %which is not a lattice pyramid. 
Assume that $\Lambda_\Delta$ is a cyclic group. Then $m$ must be $3$ or $4$ or $6$. 
Moreover, the generator of $\Lambda_\Delta$ looks as follows: 
\begin{itemize}
\item $(1/3,\ldots,1/3) \in [0,1)^{3k}$ or its inverse when $m=3$; 
\item $(\underbrace{1/4,\ldots,1/4}_{2k},\underbrace{1/2,\ldots,1/2}_k) \in [0,1)^{3k}$ 
or its inverse when $m=4$; 
\item $(\underbrace{1/6,\ldots,1/6}_{k},\underbrace{1/3,\ldots,1/3}_k,\underbrace{1/2,\ldots,1/2}_k) \in [0,1)^{3k}$ 
or its inverse when $m=6$. 
\end{itemize}
\end{Corollary}
\begin{proof}
Let $x \in \Lambda_\Delta$ be the unique element with $\heit(x)=2k$. 
By Lemma \ref{x} (b), $x$ and its inverse must be a generator of $\Lambda_\Delta$. 
On the other hand, by Proposition \ref{order} and $m \ge 3$, $m$ is $3$ or $4$ or $6$. 
The form of $x$ follows directly from Proposition \ref{order}. 
\end{proof}

\subsection{Proof of Theorem \ref{main}}

Let $\Delta$ be an empty simplex whose $h^*$-polynomial equals $1+(m-2)t^k+t^{2k}$ for given integers $m \geq 3$ and $k \geq 2$. 

\smallskip

By Corollary \ref{ccc} we can assume that $\Lambda_\Delta$ is not cyclic. 
Namely, we assume that there is a group isomorphism 
$$\phi : \Lambda_\Delta \rightarrow \ZZ/m_1\ZZ \times \cdots \times \ZZ/m_\ell\ZZ,$$ 
where $\ell \geq 2$, $m_i \in \ZZ_{\geq 2}$ and $m_i$ divides $m_{i+1}$ for each $1 \leq i \leq \ell-1$. 

Let $x \in \Lambda_\Delta$ be the unique element with $\heit(x)=2k$. 
Then there is $x^{(i)} \in \ZZ/m_i\ZZ$ for each $1 \leq i \leq \ell$ 
such that $\phi(x)=(x^{(1)},\ldots,x^{(\ell)}) \in \ZZ/m_1\ZZ \times \cdots \times \ZZ/m_\ell\ZZ$. 
Let $S=\{i \in \{1,\ldots,\ell\} : x^{(i)} \not=0\}$. Then $S \not=\emptyset$. 

We will split the proof into two cases. 

\subsection{The case $\ell \geq 3$} 

First, we consider the case $\ell \geq 3$. 

Assume that $\sharp S > 1$. Then there are $q$ and $q'$ in $S$ such that $q \not= q'$. Let 
$$G=\phi^{-1}(\ZZ/m_1\ZZ \times \cdots \times \ZZ/m_{q-1}\ZZ \times \{0\} \times \ZZ/m_{q+1}\ZZ \times \cdots \times \ZZ/m_\ell\ZZ).$$ 
Then $G$ is a subgroup of $\Lambda_\Delta$ not containing $x$. 
Let $\Delta_G \subset \RR^d$ be a lattice simplex such that $\Lambda_{\Delta_G}=G$. 
Since we have $\heit(y)=k$ for each $y \in G \setminus \{{\bf 0}\}$, 
the $h^*$-polynomial of $\Delta_G$ equals $1+(\sharp G -1)t^k$. Moreover, since $\ell \geq 3$, $G$ is not cyclic. 
Although $\Delta_G$ might be a lattice pyramid, the structure of $\Delta_G$ (equivalently, $G$) 
is known by Theorem \ref{key1} or \ref{key2}. 
Since $G$ is not cyclic, $\Delta_G$ is the case of Theorem \ref{key2}. In particular, 
there are a prime number $p$ and a positive integer $r$ 
such that $G \cong (\ZZ/p\ZZ)^r$. Hence, $m_1=\cdots=m_{q-1}=m_{q+1}=\cdots=m_\ell=p$ and $r=\ell-1$. Similarly, let 
$$G'=\phi^{-1}(\ZZ/m_1\ZZ \times \cdots \times \ZZ/m_{q'-1}\ZZ \times \{0\} \times \ZZ/m_{q'+1}\ZZ \times \cdots \times \ZZ/m_\ell\ZZ).$$ 
Then the same discussion as above shows that there is a prime number $p'$ such that $m_1=\cdots=m_{q'-1}=m_{q'+1}=\cdots=m_\ell=p'$. 
Since $\ell \geq 3$ and $q \not= q'$, we conclude that $m_1=\cdots=m_\ell=p(=p')$, that is, $\Lambda_\Delta \cong (\ZZ/p\ZZ)^\ell$. 
Moreover, since the order of $x \in \Lambda_\Delta$ is 2 or 3 or 4 or 6 by Proposition \ref{order}, $p$ should be 2 or 3. Therefore, 
$$\Lambda_\Delta \cong (\ZZ/2\ZZ)^\ell \text{ or } \Lambda_\Delta \cong (\ZZ/3\ZZ)^\ell.$$
In each case, there is another isomorphism $\phi':\Lambda_\Delta \rightarrow (\ZZ/b\ZZ)^\ell$, 
where $b=2$ or $b=3$, such that $\phi'(x)=(0,\ldots,0,1) \in (\ZZ/b\ZZ)^\ell$. 

Hence, we can assume the case $\sharp S=1$. 
By Lemma \ref{x} (b) and $\sharp S=1$, $\phi(x)$ generates one direct factor of $\phi(\Lambda_\Delta)$ 
and so does $\phi(-x)$. For the remaining direct factors, the same discussions as above can be applied. 
Therefore, $\phi(\Lambda_\Delta)$ must be one of the following (non-cyclic) groups: 
\begin{itemize}
\item[(i)] $(\ZZ / 2\ZZ)^\ell$; %in this case, the order of $x$ is 2; 
\item[(ii)] $(\ZZ / 2\ZZ)^{\ell-1} \times \ZZ/4\ZZ$; %in this case, the order of $x$ is 4; 
\item[(iii)] $(\ZZ / 2\ZZ)^{\ell-1} \times \ZZ/6\ZZ$; %in this case, the order of $x$ is 6; 
\item[(iv)] $(\ZZ / 3\ZZ)^\ell$; %in this case, the order of $x$ is 3; 
\item[(v)] $(\ZZ / 3\ZZ)^{\ell-1} \times \ZZ/6\ZZ$. %in this case, the order of $x$ is 6; 
\end{itemize}
Here, we assume that $\phi(x)$ belongs to the last direct factor. 

By the discussions below, we verify the cases (i) and (iv) can happen 
but the cases (ii), (iii) and (v) never happen.

\bigskip

\noindent
The case (i): Let us consider the subgroup $G'=\phi^{-1}((\ZZ/2\ZZ)^{\ell-1} \times \{0\})$ of $\Lambda_\Delta$, where $x \not\in G'$. 
Then it follows that we have $\heit(y)=k$ for each $y \in G' \setminus \{{\bf 0}\}$. 
By Theorem \ref{key2}, we know the system of generator of $G'$ as follows: 
let $\Delta'$ be the lattice simplex of dimension $d' \leq d$ which is not a lattice pyramid 
such that $\Lambda_{\Delta'} = G'$ after taking $(d-d')$-repeated lattice pyramids. 
Then the system of generators of $\Lambda_{\Delta'}$ is the set of the row vectors of the matrix 
$$(B_{\ell-1}^{(2)},\ldots,B_{\ell-1}^{(2)}),$$ 
where $A_{\ell-1}^{(2)} \in \{0,1\}^{(\ell-1) \times (2^{\ell-1}-1)}$ 
is the generator matrix of the simplex code over $\FF_2$ of dimension $(\ell-1)$ with block length $(2^{\ell-1}-1)$ 
and $B_{\ell-1}^{(2)} \in \left\{0,\frac{1}{2}\right\}^{(\ell-1) \times (2^{\ell-1}-1)}$ 
is the matrix all of whose entries are divided by $2$ from those of $A_{\ell-1}^{(2)}$, 
and where $B_{\ell-1}^{(2)} \in \left\{0,\frac{1}{2}\right\}^{(\ell-1) \times 2^{\ell-1}}$ is repeated $k/2^{\ell-3}$ times. 

Let $a=k/2^{\ell-3}$. Then $k=2^{\ell-3}a$ and $a \geq 1$. By Theorem \ref{key2}, 
we know the relation $$2^{\ell-2}(d'+1)=2k(2^{\ell-1}-1)=2^{\ell-2}a(2^{\ell-1}-1).$$ 
Thus $d'+1=a(2^{\ell-1}-1)$.

On the other hand, since the order of $x$ is 2 in this case, 
we have $\sharp \supp(x)=4k=d+1$ by Proposition \ref{order}. 
Therefore, $$d+1-(d'+1)=2^{\ell-1}a-a(2^{\ell-1}-1)=a.$$ 
Consequently, in this case, we have $m=2^\ell$, $k=2^{\ell-3}a$ and $d=4k-1=2^{\ell-1}a-1$ 
with $a \geq 1$ and $\ell \geq 3$ and 
the system of generators of $\Lambda_\Delta$ is the set of row vectors of the matrix 
\begin{align*}
\begin{pmatrix}
&(B_{\ell-1}^{(2)},0) &(B_{\ell-1}^{(2)},0) &\cdots &(B_{\ell-1}^{(2)},0) \\
&1/2 &1/2 &\cdots &1/2
\end{pmatrix} \in [0,1)^{\ell \times 4k} 
\end{align*}
up to permutation of the columns. This is the case (b) of Theorem \ref{main}. 

\bigskip

\noindent
The cases (ii) and (iii): 
Let $G'$ be the same thing as the case (i) above. 

Since the order of $x$ is 4 or 6, we have $d+1=3k$ by Proposition \ref{order}. Take $y \in G' \setminus \{{\bf 0}\}$. 
Since the order of $y$ is 2, we have $y=(\underbrace{1/2,\ldots,1/2}_{2k},\underbrace{0,\ldots,0}_k) \in [0,1)^{3k}$ after reordering. 
By $\sharp \supp(x+y)=2k$, $\sharp \{i \in \supp(x) \cap \supp(y) : x_i=1/2\}$ should be $k$. 
Similarly, for $y' \in G' \setminus \{{\bf 0}\}$ with $y \not=y'$, one has $\sharp \{i \in \supp(x) \cap \supp(y') : x_i=1/2\}=k$. 
Recall that $\sharp\{i \in \supp(x) : x_i=1/2\}=k$ by Proposition \ref{order}. 
Thus, $\sharp \supp(x+y+y')=3k$, a contradiction.

\bigskip

\noindent
The case (iv): 
Let us consider the subgroup $G'=\phi^{-1}((\ZZ/3\ZZ)^{\ell-1} \times \{0\})$ of $\Lambda_\Delta$, where $x \not\in G'$. 
Let $\Delta'$ be a lattice simplex of dimension $d' \leq d$ which is not a lattice simplex such that 
$\Lambda_{\Delta'}=G'$ after taking $(d-d')$-repeated lattice pyramids. 
Then the system of generators of $\Lambda_{\Delta'}$ is the set of the row vectors of the matrix 
$$((B_{\ell-1}^{(3)},-B_{\ell-1}^{(3)}),\ldots,(B_{\ell-1}^{(3)},-B_{\ell-1}^{(3)})),$$ 
where $A_{\ell-1}^{(3)} \in \{0,1,2\}^{(\ell-1) \times (3^{\ell-1}-1)/2}$ 
is the generator matrix of the simplex code over $\FF_3$ of dimension $(\ell-1)$ with block length $(3^{\ell-1}-1)/2$ and 
$B_{\ell-1}^{(3)} \in \left\{0,\frac{1}{3},\frac{2}{3}\right\}^{(\ell-1) \times (3^{\ell-1}-1)/2}$ 
(resp. $-B_{\ell-1}^{(3)} \in \left\{0,\frac{2}{3},\frac{1}{3}\right\}^{(\ell-1) \times (3^{\ell-1}-1)/2}$) is the matrix all of whose entries are 
divided by $3$ from those of $A_{\ell-1}^{(3)}$ (resp. $-A_{\ell-1}^{(3)}$), 
and where in above matrix $(B_{\ell-1}^{(3)},-B_{\ell-1}^{(3)}) \in \left\{0,\frac{1}{3},\frac{2}{3}\right\}^{(\ell-1) \times (3^{\ell-1}-1)}$ 
is repeated $(k/3^{\ell-2})$ times. 

Let $a=k/3^{\ell-2}$. Then $k= 3^{\ell-2}a$ and $a \geq 1$. By Theorem \ref{key2}, we know the relation 
$$(3^{\ell-1}-3^{\ell-2})(d'+1)=2k(3^{\ell-1}-1)=2 \cdot 3^{\ell-2}a(3^{\ell-1}-1).$$ 
Thus $d'+1=a(3^{\ell-1}-1)$. 

Since the order of $x$ is 3, we have $\sharp \supp(x)=3k=d+1$ by Proposition \ref{order}. 
Therefore, $$d+1-(d'+1)=3^{\ell-1}a-a(3^{\ell-1}-1)=a.$$ 
Consequently, in this case, we have $m=3^\ell$, $k=3^{\ell-2}a$ and $d=3k-1=3^{\ell-1}a-1$ 
with $a \geq 1$ and $\ell \geq 3$ and 
the system of generators of $\Lambda_\Delta$ is the set of row vectors of the matrix 
\begin{align*}
\begin{pmatrix}
&(B_{\ell-1}^{(3)},-B_{\ell-1}^{(3)},0) &(B_{\ell-1}^{(3)},-B_{\ell-1}^{(3)},0) &\cdots &(B_{\ell-1}^{(3)},-B_{\ell-1}^{(3)}.0) \\
&1/3 &1/3 &\cdots &1/3
\end{pmatrix} \in [0,1)^{\ell \times 3k} 
\end{align*}
up to permutation of the columns. 
This is the case (c) of Theorem \ref{main} with $\ell \geq 3$. 

\bigskip

\noindent
The case (v): 
Let $G'$ be the same thing as the case (iv). 

Take $y \in G' \setminus \{{\bf 0}\}$. Then $y=(\underbrace{1/3,\ldots,1/3}_k,\underbrace{2/3,\ldots,2/3}_k,\underbrace{0,\ldots,0}_k)$ 
after reordering. Since $\sharp \supp(x+y)=2k$, $\sharp \{i \in \supp(x) \cap \supp(y) : x_i=2/3,y_i=1/3\}$ should be $k$. 
Thus, we have $\supp(x+2y)=3k$, a contradiction. 
%Similarly, for $y' \in G' \setminus \{{\bf 0}\}$ with $y \not= y'$, 
%one has $\sharp \{i \in \supp(x) \cap \supp(y') : x_i=2/3,y_i'=1/3\}=k$. 
%Let $z=y-y' \in G'$. Since $\sharp \{i \in \supp(x) \cap \supp(z) : x_i=2/3,z_i=1/3\}=0$, 
%we have $\sharp \supp(x+z) = 3k$, a contradiction. 
%Therefore, $\Lambda_\Delta$ is never isomorphic to $(\ZZ/3\ZZ)^{\ell-1} \times \ZZ/6\ZZ$. 

\subsection{The case $\ell=2$}

Next, we consider the case $\ell=2$.

Let $G_1=\phi^{-1}(\ZZ/m_1\ZZ \times \{0\})$ and $G_2=\phi^{-1}(\{0\} \times \ZZ/m_2\ZZ)$. 
Clearly, either $G_1$ or $G_2$ does not contain $x$, say, $G_1$. 
Then we have $\heit(y)=k$ for each $y \in G_1 \setminus \{{\bf 0}\}$. 
By Theorem \ref{key1}, $G_1$ is generated by $(a_1/m_q,(m_q-a_1)/m_q,\ldots,a_k/m_q,(m_q-a_k)/m_q,0,\ldots,0) \in G_1$ 
after reordering, where $m_q=\sharp G_1$ and each $a_i$ is an integer with $0 < a_i \leq m_q/2$ which is coprime to $m_q$. 
Let $g=(a_1/m_q,(m_q-a_1)/m_q,\ldots,a_k/m_q,(m_q-a_k)/m_q,0,\ldots,0) \in [0,1)^{d+1}$. 

Let $\phi(x)=(x^{(1)},x^{(2)}) \in \ZZ/m_1\ZZ \times \ZZ/m_2\ZZ$, where $0 \leq x^{(i)} \leq m_i-1$ for $i=1,2$. 

\bigskip

\noindent
The case where the order of $x$ is $2$: 
Then $x=(1/2,\ldots,1/2) \in [0,1)^{4k}$ and $d+1=4k$ by Proposition \ref{order}. 
Since $\sharp \supp(x+g)=2k$, we obtain that $m_q=2$ and $a_i=1$ for each $i$. 

Assume that $x^{(1)}\not=0$ and $x^{(2)}\not=0$. Let $g_1$ and $g_2$ be the generators of $G_1$ and $G_2$, 
respectively, such that $x=x^{(1)}g_1+x^{(2)}g_2$. 
Since $\sharp \supp(x^{(1)}g_1)=\sharp \supp(x^{(2)}g_2)=2k$ and $x=(1/2,\ldots,1/2) \in [0,1)^{4k}$, 
$g_1$ and $g_2$ look like $(\underbrace{1/2,\ldots,1/2}_{2k},\underbrace{0,\ldots,0}_{2k})$ after reordering 
and we also have $x^{(1)}=x^{(2)}=1$. 
In particular, $\phi(\Lambda_\Delta) = (\ZZ/2\ZZ)^2$. 
Thus there is another isomorphism $\phi' : \Lambda_\Delta \rightarrow (\ZZ/2\ZZ)^2$ such that $\phi'(x)=(0,1) \in (\ZZ/2\ZZ)^2$. 
Hence we can deduce the case where $x^{(1)}=0$ or $x^{(2)}=0$.

Assume that $x^{(1)}=0$ or $x^{(2)}=0$. Then $x$ generates one direct factor of $\Lambda_\Delta$. 
Hence the system of generators of $\Lambda_\Delta$ is the set of row vectors of the matrix 
\begin{align*}
\left( \, 
\begin{array}{r@{}r@{}r r r}
\underbrace{
\begin{array}{rrr}
1/2 &\cdots &1/2 \\
1/2 &\cdots &1/2 
\end{array}
}_{2k} \; 
\underbrace{
\begin{array}{rrr}
0   &\cdots &0 \\
1/2 &\cdots &1/2 
\end{array}
}_{2k} 
\end{array}\, \right) \in [0,1)^{2 \times 4k}
\end{align*}
after reordering. This is the case (a) with $m=4$ and $d=4k-1$ of Theorem \ref{main}.

\bigskip

\noindent
The case where the order of $x$ is $3$: 
Then $x=(2/3,\ldots,2/3) \in [0,1)^{3k}$ by Proposition \ref{order}. 
Since $\sharp \supp(x+g)=2k$, we obtain $m_q=3$ and $a_i=1$ for each $i$.

Assume that $x^{(1)}\not=0$ and $x^{(2)}\not=0$. 
By the similar discussions to the above, we see that $\phi(\Lambda_\Delta) = (\ZZ/3\ZZ)^2$. 
Thus there is another isomorphism $\phi' : \Lambda_\Delta \rightarrow (\ZZ/3\ZZ)^2$ such that $\phi'(x)=(0,1) \in (\ZZ/3\ZZ)^2$. 
Hence we can deduce the case where $x^{(1)}=0$ or $x^{(2)}=0$.

Assume that $x^{(1)}=0$ or $x^{(2)}=0$. 
Then each of $x$ and $-x$ generates one direct factor of $\Lambda_\Delta$. 
Hence we obtain that the system of generators of $\Lambda_\Delta$ is the set of row vectors of the matrix 
\begin{align*}
\left( \, 
\begin{array}{r@{}r@{}r r r}
\underbrace{
\begin{array}{rrr}
1/3 &\cdots &1/3 \\
1/3 &\cdots &1/3 
\end{array}
}_k \;
\underbrace{
\begin{array}{rrr}
2/3 &\cdots &2/3 \\
1/3 &\cdots &1/3 
\end{array}
}_k \;
\underbrace{
\begin{array}{rrr}
0   &\cdots &0 \\
1/3 &\cdots &1/3 
\end{array}
}_k 
\end{array}\, \right) \in [0,1)^{2 \times 3k}. 
\end{align*}
This is the case (c) with $\ell=2$ of Theorem \ref{main}. 

\bigskip

\noindent
The case where the order of $x$ is $4$: 
Then $x=(\underbrace{3/4,\ldots,3/4}_{2k},\underbrace{1/2,\ldots,1/2}_k) \in [0,1)^{3k}$ by Proposition \ref{order}. 
Let $k_j=\sharp \{i \in \supp(x) \cap \supp(g) : x_i=3/4, g_i=j/4\}$ for $j=1,2,3$ and 
$k'=\sharp \{i \in \supp(x) \cap \supp(g) : x_i=g_i=1/2\}$. 
Since $\sharp \supp(x+g)=\sharp \supp(2x+g)=\sharp \supp(3x+g)=2k$, similar to the proof of Proposition \ref{order}, 
we obtain $k_1+k'=k_3+k'=k$ and $k_2 \geq k'$. Thus we have $k_1+k_2+k_3+k' \geq 2k$. 
On the other hand, since $\sharp \supp(g) = 2k$, we also have $k_1+k_2+k_3+k' \leq 2k$. 
Hence $k_1+k_2+k_3+k'=2k$. Moreover, since $\sharp \supp(2x+2g)=2k$, one has 
$\sharp \supp(2x+2g)= 2k-k_1-k_3=2k$. Thus $k_1=k_3=0$. Hence it follows from $k_1+k'=k_3+k'=k$ that $k_2=k'=k$. 
In particular, $g$ looks like $(\underbrace{1/2,\ldots,1/2}_{2k},\underbrace{0,\ldots,0}_k) \in [0,1)^{3k}$ after reordering 
and has order $2$. 

Assume that $x^{(1)}\not=0$ and $x^{(2)}\not=0$. 
each generator of $G_1$ and $G_2$ has order $2$, we obtain that $\Lambda_\Delta \cong (\ZZ/2\ZZ)^2$. 
However, $(\ZZ/2\ZZ)^2$ does not contain any element with order $4$, a contradiction. 

Hence $x^{(1)}=0$ or $x^{(2)}=0$. Then each of $x$ and $-x$ generates one direct factor of $\Lambda_\Delta$. 
Thus we see that the system of generators of $\Lambda_\Delta$ is the set of row vectors of the matrix 
\begin{align*}
\left( \, 
\begin{array}{r@{}r@{}r r r}
\underbrace{
\begin{array}{rrr}
1/2 &\cdots &1/2 \\
1/4 &\cdots &1/4 
\end{array}
}_k \; 
\underbrace{
\begin{array}{rrr}
0   &\cdots &0 \\
1/4 &\cdots &1/4 
\end{array}
}_k \; 
\underbrace{
\begin{array}{rrr}
1/2 &\cdots &1/2 \\
1/2 &\cdots &1/2 
\end{array}
}_k 
\end{array}\, \right) \in [0,1)^{2 \times 3k}. 
\end{align*}
This is the case (a) with $m=8$ of Theorem \ref{main}.

\bigskip

\noindent
The case where the order of $x$ is $6$: 
Then $x=(\underbrace{5/6,\ldots,5/6}_k,\underbrace{2/3,\ldots,2/3}_k,\underbrace{1/2,\ldots,1/2}_k) \in [0,1)^{3k}$ 
by Proposition \ref{order}. 
Let \begin{align*}
&k_j=\sharp \{i \in \supp(x) \cap \supp(g) : x_i=5/6, g_i=j/6\} \text{ for } j=1,2,3,4,5, \\
&k_j'=\sharp \{i \in \supp(x) \cap \supp(g) : x_i=2/3, g_i=j/3\} \text{ for } j=1,2, \\
&k''=\sharp \{i \in \supp(x) \cap \supp(g) : x_i=g_i=1/2\}. 
\end{align*}
Since $\sharp \supp(x+g)=\cdots=\sharp \supp(5x+g)=2k$, similar to the proof of Proposition \ref{order}, we see that 
$k_1+\cdots+k_5+k_1'+k_2'+k'' \geq 2k$. On the other hand, since $\sharp \supp(g)=2k$, we also have 
$k_1+\cdots+k_5+k_1'+k_2'+k'' \leq 2k$. Hence, $k_1+\cdots+k_5+k_1'+k_2'+k'' =2k$. 

Moreover, since $\sharp \supp(x+2g)=\sharp \supp(x+4g)=2k$, one also has 
$$2k=\sharp \supp(x+2g)=k+k-k_2'+k \text{ and }2k=\sharp \supp(x+4g)=k+k-k_1'+k.$$ 
Hence $k=k_1'=k_2'$. Then it follows that $2k=k_1'+k_2' \leq \sharp\{i \in \supp(x) : x_i=2/3\} = k$, a contradiction. 

Therefore, we conclude that the order of $x$ is never $6$ when $\Lambda_\Delta$ has exactly two direct factors. 
This finishes the proof of Theorem \ref{main}. 

{\small 
\bibliographystyle{plain}
\bibliography{symmetric_hvector-arxiv-version}

\begin{thebibliography}{10}

\bibitem{BBBK11}
M.~Barile, D.~Bernardi, A.~Borisov, and J.-M. Kantor.
\newblock On empty lattice simplices in dimension 4.
\newblock {\em Proc. Amer. Math. Soc.}, 139(12):4247--4253, 2011.

\bibitem{BatyrevJuny}
Victor Batyrev and Dorothee Juny.
\newblock Classification of {G}orenstein toric del {P}ezzo varieties in
  arbitrary dimension.
\newblock {\em Mosc. Math. J.}, 10(2):285--316, 478, 2010.

\bibitem{BN07}
Victor Batyrev and Benjamin Nill.
\newblock Multiples of lattice polytopes without interior lattice points.
\newblock {\em Mosc. Math. J.}, 7(2):195--207, 349, 2007.

\bibitem{Bat94}
V.V. Batyrev.
\newblock Dual polyhedra and mirror symmetry for {C}alabi-{Y}au hypersurfaces
  in toric varieties.
\newblock {\em J. Algebraic Geom.}, 3(3):493--535, 1994.

\bibitem{Bat06}
V.V. Batyrev.
\newblock Lattice polytopes with a given {$h^*$}-polynomial.
\newblock In {\em Algebraic and geometric combinatorics}, volume 423 of {\em
  Contemp. Math.}, pages 1--10. Amer. Math. Soc., Providence, RI, 2006.

\bibitem{BatyrevHof1}
V.V. Batyrev and J.~Hofscheier.
\newblock A generalization of a theorem of {G. K.} {White}.
\newblock arXiv:1004.3411, 2010.

\bibitem{BatyrevHof2}
V.V. Batyrev and J.~Hofscheier.
\newblock Lattice polytopes, finite abelian subgroups in
  $\text{SL}(n,\mathbb{C})$ and coding theory.
\newblock arXiv:1309.5312, 2013.

\bibitem{BN08}
V.V. Batyrev and B.~Nill.
\newblock Combinatorial aspects of mirror symmetry.
\newblock In {\em Integer points in polyhedra}, volume 452 of {\em Contemp.
  Math.}, pages 35--66. Amer. Math. Soc., 2008.

\bibitem{BeckRobins}
M.~Beck and S.~Robins.
\newblock {\em Computing the continuous discretely}.
\newblock Undergraduate Texts in Mathematics. Springer, 2007.

\bibitem{Santos1}
M.~Blanco and F.~Santos.
\newblock Lattice 3-polytopes with five lattice points.
\newblock arXiv:1409.6701, 2014.

\bibitem{Santos2}
M.~Blanco and F.~Santos.
\newblock Lattice 3-polytopes with six lattice points.
\newblock arXiv:1501.01055, 2015.

\bibitem{HNP09}
C.~Haase, B.~Nill, and S.~Payne.
\newblock Cayley decompositions of lattice polytopes and upper bounds for
  {$h^*$}-polynomials.
\newblock {\em J. Reine Angew. Math.}, 637:207--216, 2009.

\bibitem{Haase}
Christian Haase and Josef Schicho.
\newblock Lattice polygons and the number {$2i+7$}.
\newblock {\em Amer. Math. Monthly}, 116(2):151--165, 2009.

\bibitem{Henk}
Martin {Henk} and Makoto {Tagami}.
\newblock {Lower bounds on the coefficients of Ehrhart polynomials.}
\newblock {\em {Eur. J. Comb.}}, 30(1):70--83, 2009.

\bibitem{SmallVol}
Takayuki Hibi, Akihiro Higashitani, and Yuuki Nagazawa.
\newblock Ehrhart polynomials of convex polytopes with small volumes.
\newblock {\em European J. Combin.}, 32(2):226--232, 2011.

\bibitem{ShiftedSym}
Akihiro Higashitani.
\newblock Shifted symmetric {$\delta$}-vectors of convex polytopes.
\newblock {\em Discrete Math.}, 310(21):2925--2934, 2010.

\bibitem{PrimeVol}
Akihiro Higashitani.
\newblock Ehrhart polynomials of integral simplices with prime volumes.
\newblock {\em Integers}, 14:Paper No. A45, 15, 2014.

\bibitem{Joswig}
Michael {Joswig} and Katja {Kulas}.
\newblock {Tropical and ordinary convexity combined.}
\newblock {\em {Adv. Geom.}}, 10(2):333--352, 2010.

\bibitem{KS00}
M.~Kreuzer and H.~Skarke.
\newblock Complete classification of reflexive polyhedra in four dimensions.
\newblock {\em Adv. Theor. Math. Phys.}, 4(6):1209--1230, 2000.

\bibitem{LZ91}
Jeffrey~C. {Lagarias} and G\"unter~M. {Ziegler}.
\newblock {Bounds for lattice polytopes containing a fixed number of interior
  points in a sublattice.}
\newblock {\em {Can. J. Math.}}, 43(5):1022--1035, 1991.

\bibitem{Nil08}
B.~Nill.
\newblock Lattice polytopes having {$h^\ast$}-polynomials with given degree and
  linear coefficient.
\newblock {\em European J. Combin.}, 29(7):1596--1602, 2008.

\bibitem{Scott}
P.~R. Scott.
\newblock On convex lattice polygons.
\newblock {\em Bull. Austral. Math. Soc.}, 15(3):395--399, 1976.

\bibitem{Seb99}
A.~Seb{\H{o}}.
\newblock An introduction to empty lattice simplices.
\newblock In {\em Integer programming and combinatorial optimization ({G}raz,
  1999)}, volume 1610 of {\em Lecture Notes in Comput. Sci.}, pages 400--414.
  Springer, 1999.

\bibitem{Treutlein}
Jaron {Treutlein}.
\newblock {Lattice polytopes of degree 2.}
\newblock {\em {J. Comb. Theory, Ser. A}}, 117(3):354--360, 2010.

\bibitem{White}
G.~K. White.
\newblock Lattice tetrahedra.
\newblock {\em Canad. J. Math.}, 16:389--396, 1964.

\end{thebibliography}
}
\end{document}